\documentclass[12pt]{amsart}
\input xy
\xyoption{all}
\usepackage{amssymb}
\newcommand{\bP}{{\mathbb P}}

\newcommand{\bZ}{{\mathbb Z}}

\newtheorem{theorem}{Theorem}[section]
\newtheorem{lemma}[theorem]{Lemma}
\newtheorem{proposition}[theorem]{Proposition}
\newtheorem{corollary}[theorem]{Corollary}

\numberwithin{equation}{section}

\begin{document}
%\today

\title{Very ample linear series on quadrigonal curves}

\author{Marc Coppens and Gerriet Martens}

\address{KU Leuven, Technologiecampus Geel, Departement Elektrotechniek (ESAT), Kleinhoefstraat 4, B-2440 Geel, Belgium}
\email{marc.coppens@kuleuven.be}

\address{Department Mathematik, Univ. Erlangen-N\"urnberg\\
Cauerstr. 11, D-91058 Erlangen, Germany}
\email{martens@mi.uni-erlangen.de}


\begin{abstract}
In this paper we describe all complete and very ample linear series on 4-gonal curves.
\end{abstract}

\maketitle

\section{Introduction}

The complete and very ample linear series on 2- and 3-gonal curves are described in \cite{M1}. Here we deal with the 4-gonal case, i.e. with quadrigonal curves.

Let $C$ denote a quadrigonal curve of genus $g$; then $g \ge 5$. In \cite{CM1} we gave a typification of all {\it non-trivial linear series} on $C$, i.e. of those complete and base point free linear series $|D|$ on $C$ for which both $h^0(D)$ and $h^1(D)$ are at least 2. Writing (as is classical) $|D| = g_d^r$ with deg$(D) = d$ and $h^0(D) = r+1$ we recall that, by the Riemann-Roch theorem, $h^1(D) = r'+1$ with $r'+1 := g-d+r = h^0(K_C - D)$, the index of speciality of $D$. (Here $K_C$ denotes a canonical divisor of $C$, of course.) 

Fixing a pencil $g_4^1$ on $C$ we call a non-trivial $g_d^r$ on $C$ of

- {\it type 1} if $g_d^r = rg_4^1 = |rg_4^1|$;

- {\it type 2} if $|K_C - g_d^r| = r'g_4^1 + F$ ($r' = g-1-d+r$), for an effective divisor $F$ of $C$ (i.e. $|K_C - g_d^r|$

is of type 1 away from its base locus $F$; note that $2g-2-d =$ deg$(K_C - g_d^r) \ge 4r'$ holds which

means that $3d \ge 2g-2+4r$);

- {\it type 3} if $g_d^r = |(r-1)g_4^1 + E|$ for some effective divisor $E$ of $C$ (and we add that $h^0(E) \ge 2$

implies that $r=1$ or that $g_d^r$ is also of type 1).

More precisely, a $g_d^r$ described above is said to be of type 1 resp. type 2 resp. type 3 w.r.t. the chosen $g_4^1$. However, for $g>9$ there is only one pencil of degree $4$ on $C$ (whence there is no ambiguity in this typification, then) unless $C$ is bi-elliptic. The series of type 1, 2, 3 are very useful as the constituents of all non-trivial linear series on the quadrigonal curve $C$ (\cite{CM1}, 1.10).\\

In this note we restrict to non-trivial and {\it very ample} linear series on $C$. According to \cite{M2}, section 3 such series are of type 2 unless $3|g$ and $C$ is special w.r.t. moduli of quadrigonal curves of genus $g$. It is our aim here to give, for those "exceptional" curves, a {\it  description of all non-trivial and very ample linear series which are not of type 2}. It turns out (Theorem \ref{Theorem}) that this description is completely controlled by the smallest scrollar invariant $e_1$ of the chosen $g_4^1$, defined by $e_1 + 1 =$ Max$\lbrace n:$ dim$|ng_4^1| = n \rbrace $. For $g \ge 8$ a quadrigonal curve $C$ has at most two or - if $C$ is bi-elliptic - infinitely many pencils of degree 4, and  $e_1$ does not depend on the choice of such a pencil; cf. \cite{C}. So $e_1$ is at least for $g \ge 8$ an invariant of a quadrigonal curve which we call here  the {\it scrollar invariant of $C$}, then.

Let $V = V_C(g_4^1)$ denote the set of non-trivial and very ample linear series on $C$ which are not of type 2, and assume that $V \ne \emptyset$. In Theorem \ref{Theorem} we describe $V$ and show that $\char35 V = e_1 + 1$ resp. $\char35V = e_1$ if $V$ contains resp. does not contain a series of type 3; in a sense, $V$ is "as small as possible". Furthermore, we indicate "how special" curves with $V \ne \emptyset$ are (Corollary \ref{Cor5}). The main reason for these results is a fact observed in \cite{CC} roughly stating that $C$ can only lie on Hirzebruch surfaces $X_e$ all having the same invariant $e$.

\section{Preliminaries}

Let $C$ denote a quadrigonal curve of genus $g$. Then $g \ge 5$. Fix a pencil of degree 4 on $C$ which we denote by $g_4^1$. If $C$  is bi-elliptic it has no non-trivial and very ample linear series (e.g., \cite{M1}, (5)). On the other hand,

\begin{lemma} \label{Lemma1}
Let $g \ge 9$, and assume that $C$ is not bi-elliptic. Then $C$ has non-trivial and very ample linear series.
\end{lemma} 

\begin{proof}
Assume that $|K_C - g_4^1|$ is not very ample. Then there are two points $P$, $Q$ of $C$ such that $|g_4^1 + P + Q|$ is a $g_6^2$. Since $C$ has no $g_5^2$ this net of degree 6 is base point free, and since $C$ is neither trigonal nor bi-elliptic it is also simple, but, by its construction, not very ample. Hence we have $g \le 9$, i.e. $g=9$. Then $g_{10}^4 := |K_C - g_6^2|$ is easily seen to be base point free and simple which implies, by Castelnuovo's genus bound (\cite{A}, 3.3, or \cite{ACGH}, III, 2), that $g \le 9$; so $g=9$, i.e. the induced curve of degree 10 in ${\bP}^4$ is extremal (in the sense of \cite{ACGH}, III, 2) and therefore smooth. Hence the $g_{10}^4$ is very ample.
\end{proof}

We just saw that $C$ in Lemma \ref{Lemma1} has for $g>9$ a very ample type 2 series (namely $|K_C - g_4^1|$); more generally, $C$ has for $g>15$  very ample series $g_{2g-6-i}^{g-4-i}$ of type 2 for all $i = 0, ... , g-7$ unless it is a double covering of a curve of genus (1 or) 2 (\cite{CKM}, Example 2.6).

Recall that $V$ denotes the set of non-trivial and very ample linear series on $C$ which are not of type 2 (w.r.t. our chosen $g_4^1$). The series in $V$ could be called the "unexpected very ample series" on $C$ since for the general quadrigonal curve of genus $g$ they don't exist (\cite{M2}, 3.3). In \cite{M2}, 2.5 and 3.1 the following result is proved:

\begin{proposition} \label{Prop1} 
Assume that $V \ne \emptyset$. Then 3 divides $g$, and all series in $V$ have the same Clifford index $\frac{g}{3}-1$. For $L \in V$ we have dim$|L - g_4^1| =$ dim$(L) - 2$, and $C$ lies via $L$ on a rational normal scroll of degree dim$(L) - 1$ in ${\bP}^{dim(L)}$ if dim$(L) \ge 3$. (For a net $L \in V$ the curve $C$ obviously is a smooth plane quintic.)
\end{proposition}

{\it Example} 1: Let $V \ne \emptyset$. Then $3|g$ and $g\ge 5$. Let $g=6$. For $L \in V$ we have cliff$(L) = \frac{g}{3}-1 = 1$; hence $C$ is a quadrigonal curve of Clifford index 1, i.e. a smooth plane quintic and so $V = \lbrace g_5^2 \rbrace$. The $g_5^2$ is, in fact, the only non-trivial and very ample linear series on $C$; it induces infinitely many pencils of degree 4 on $C$, all of form $|g_5^2 - P|$ for a point $P$ of $C$, and the scrollar invariant $e_1 = 0$ does not depend on the choice of the pencil $g_4^1$ on $C$ since $|K_C| = |2g_5^2|$ and so $|K_C - 2g_4^1| = |2(g_5^2 - g_4^1)| \ne \emptyset$, i.e. dim$|2g_4^1| = 3$.

Assume that $g=9$. Then we claim that $C$ is a plane sextic with exactly one singularity (a double point) or a complete intersection of a quadric and a quartic surface in ${\bP}^3$, and $\char35 V$ is the number of pencils of degree 4 on $C$. In fact, let $L$ be the linear series of smallest degree in $V$. By the Proposition \ref{Prop1} $L$ is one of the series $g_8^3$, $g_{10}^4$, $g_{12}^5$ of Clifford index 2 on $C$. Assume that $L = g_8^3$. Then, by Castelnuovo's genus bound, $C$ is via $L$ an extremal curve of degree 8 in ${\bP}^3$ whence $C$ is a smooth complete intersection of a quadric $Q$ and a quartic surface. If $Q$ is smooth $C$ has two different pencils $g_4^1$ and $h_4^1$ of degree 4; they are cut out on $C$ by the two rulings of $Q$ and $L = |g_4^1 + h_4^1|$; if $Q$ is a cone then the $g_4^1$ on $C$ is unique and $L = |2g_4^1|$. Note that $L$ is half-canonical (i.e. $|K_C| = |2L|$) and unique. If a $g_{10}^4$ exists in $V$ then, by the Proposition \ref{Prop1}, $C$ has a net $g_6^2 := |g_{10}^4 - g_4^1|$ whence $C$ is a divisorial complete curve which, however, is not true (\cite{CM2}). A $g_{12}^5$ on $C$ is $|K_C - g_4^1|$ or - if $Q$ is smooth - $|K_C - h_4^1|$, and $|K_C - g_4^1|$ is of type 2 (w.r.t. our chosen $g_4^1$) but $|K_C - h_4^1|$ (though being of type 2 w.r.t. $h_4^1$) is not. Hence we have $V = \lbrace L, |K_C - h_4^1| \rbrace = \lbrace |h_4^1 + g_4^1|, |h_4^1 + 2 g_4^1| \rbrace$, $e_1 = 2$ (since $|K_C| = |2(g_4^1 + h_4^1)|$ implies that $|K_C - 3g_4^1| = \emptyset$), resp. $V = \lbrace L \rbrace$, $e_1 = 0$, according to $Q$ is smooth or not. Since $W_7^2 = g_8^3 - W_1$ these curves are plane septics having exactly two different triple points resp. one triple point with another one infinitely near to it whence they depend on 18 resp. 17 moduli; so they lie in a locus of codimension 3 resp 4 in the moduli space of 4-gonal curves of genus 9. (Cf. \cite{C}.)

Assume that $L = g_{10}^4$. We observed above that $C$ is then a plane sextic with a single singularity (an ordinary cusp or node) whence $C$ depends on 19 moduli. $C$ likewise is an extremal curve of degree 10 in ${\bP}^4$. So $g_4^1$ and $L$ are unique (the latter e.g. by \cite{A}, section 6), $|K_C - L| = |L - g_4^1|$ is the only $g_6^2$ on $C$, $g_{12}^5 = |K_C - g_4^1|$ is of type 2, $V =\lbrace L \rbrace$, $e_1 = 1$ (since $|K_C - 3g_4^1| = |2(g_6^2 - g_4^1)| \ne \emptyset$).

As a digression we add the remarkable fact that $L = g_{10}^4 \in V$ is the only non-trivial and very ample linear series on our plane sextic $C$. In fact, any other such series $g_d^r$ on $C$ is (as we saw) of type 2 whence $3d \ge 2g-2+4r \ge 28$, i.e. $d \ge 10$, and so $g_d^r$ is among $|K_C - g_4^1|$, $|K_C - g_4^1 - P|$, $|K_C - g_4^1 - P - Q|$ for two points $P$ and $Q$ of $C$. Here $|K_C - g_4^1|$ is base point free (since $C$ has no $g_{11}^5$) but not very ample since $|(K_C - g_4^1) - L| = |L - 2g_4^1| = |g_6^2 - g_4^1| \ne \emptyset$; so $|K_C - g_4^1 - P| = g_{11}^4$, and since $|K_C - g_4^1 - P - Q|$ cannot be a $g_{10}^4$ (this would contradict the inequality $3d \ge 2g-2+4r$) it is a $g_{10}^3$. To see that the latter two series of degree 11 and 10 likewise are not very ample we use the simple fact that if $P_0$ is not a base point of a complete linear series $L_0$ and $|L_0 - P_0|$ is very ample then so is $L_0$.

Finally, assume that $L = g_{12}^5$. Then, by Proposition \ref{Prop1}, $|L - g_4^1| = g_8^3$ which is simple and, then, very ample; since 3deg$|L-g_4^1| = 24 < 28 = 2g-2+4$dim$|L-g_4^1|$ this series cannot be of type 2. But this contradicts the minimality of deg$(L)$.

\vspace{1ex}

{\it Example} 2: For every integer $r \ge 2$ there exists an extremal curve $C$ of degree $d = 3r-1$ in ${\bP}^r$, and for $r\ge 3$ we can find such a curve on a scroll of degree $r-1$ in ${\bP}^r$ (\cite{ACGH}, III, 2). Then $C$ is quadrigonal of genus $g = 3r$, the embedding linear series $g_d^r$ is half-canonical, and it lies in $V$ since $3d < 2g-2+4r$. Hence for every integer $g \ge 5$ with $3|g$ there are quadrigonal curves of genus $g$ with $V \ne \emptyset$.

It will turn out (Corollary \ref{Cor4}) that Example 2 is rather extensive.

\section{The Theorem}

Our main result is the

\begin{theorem} \label{Theorem}
Let $C$ be a quadrigonal curve of genus $g$ with $V \ne \emptyset$ and scrollar invariant $e_1$. Fix some $g_4^1$ on $C$. Then there exists a unique non-trivial linear series $g_{d_0}^{r_0}$ of degree $d_0 = g-1-2e_1$ and dimension $r_0 = \frac{g}{3} - e_1$ on $C$, simple for $r_0 \ge 2$, and for $r_0 = 1$ different from our chosen $g_4^1$ and not compounded of the same involution as the $g_4^1$. This series  is of type 3, i.e.  $g_{d_0}^{r_0} = |(r_0 - 1)g_4^1 + D_0|$ for some effective divisor $D_0$ of $C$, and it is the generator of $V$, in the following sense: All series in $V$ are exactly the $g_{d_0}^{r_0}$ provided that deg$(D_0) \le 1$, and the series $g_{d_0 + 4i}^{r_0 + 2i} := |g_{d_0}^{r_0} + ig_4^1|$ for all $i = 1, ... , e_1$. (In particular, $\char35 V = e_1 + 1$ resp. $\char35 V = e_1$ if and only if deg$(D_0) \le 1$ resp. deg$(D_0) > 1$.) The set $\lbrace g_{d_0}^{r_0} \rbrace \cup V$ is closed under dualization, more precisely: $|K_C - (g_{d_0}^{r_0} + ig_4^1)| = |g_{d_0}^{r_0} + (e_1 - i)g_4^1|$ for $i = 0, ... , e_1$.
\end{theorem} 

Before we prove this Theorem we add another Example.

\vspace{1ex}

{\it Example} 3: $\char35 V = 1$ if and only if $C$ is a smooth plane quintic ($g=6$), or a smooth complete intersection of a quadric cone and a quartic surface in ${\bP}^3$, or a plane sextic with exactly one singular point, an ordinary cusp or node; so $g=9$ in the last two cases.

In fact, in Example 1 we observed already that these curves satisfy $\char35 V = 1$. Conversely, by the Theorem $\char35 V = 1$ implies that $e_1 = 0$ or $e_1 = 1$. If $e_1 = 0$ then $|2g_4^1| = g _8^3$ is simple (recall that a bi-elliptic curve has $V = \emptyset$) whence $g \le 9$, i.e. $g=6$ or $g=9$, and again by Example 1 we are done. Let $e_1 = 1$. Then $i=1$ in the Theorem, the generator $g_{d_0}^{r_0}$ of $V$ is not contained in $V$, and so $V = \lbrace |g_{d_0}^{r_0} + g_4^1| \rbrace$ with $g-3 = g-1-2e_1 = d_0 = 4(r_0 - 1) +$ deg$(D_0) = 4(\frac{g}{3} - 2) +$ deg$(D_0) \ge 4(\frac{g}{3} - 2) + 2 = \frac{4g}{3} - 6$, i.e. $g \le 9$, and it follows $g=9$, $g_{d_0}^{r_0} = g_6^2$ simple, and $V = \lbrace g_{10}^4 \rbrace$.

Note that if $C$ has for $g=9$ two different pencils $g_4^1$ and $h_4^1$ of degree 4 then we have $V_C(g_4^1) = \lbrace |h_4^1 + g_4^1|, |h_4^1 + 2 g_4^1| \rbrace$, $V_C(h_4^1) = \lbrace |g_4^1 + h_4^1|, |g_4^1 + 2 h_4^1| \rbrace$, and $|g_4^1 + h_4^1|$ is the only non-trivial and very ample linear series on $C$ which is not of type 2 w.r.t. some pencil of degree 4.

\begin{proof}
Take a $g_d^r \in V$; for $r=2$ we have $V = \lbrace g_5^2 \rbrace$, and we are done, with $e_1 = 0$. So let $r \ge 3$. By Proposition  \ref{Prop1} $g_d^r$ embeds $C$ in a scrollar surface $S$ of degree $r-1$ in ${\bP}^r$ which comes from a rational ruled surface (Hirzebruch surface) $X = X_e$ of invariant $e \ge 0$ by a complete linear series $|C_0 + nf|$ of sections of $X$, for an integer $n \ge e$. We use the notation of Hartshorne's textbook \cite{H}, V, 2 and observe that $|C_0 + nf|$ is very ample (i.e. $S \subset {\bP}^r$ is smooth) for $n>e$ and still base point free for $n=e$ (whence $S$ is a cone over a rational normal curve in ${\bP}^{r-1}$); note that $r = 2n - e + 1$ (\cite{H}, V, 2.17, 2.19), and in the case $n=e$ we have $e \ge 2$ since $r \ge 3$. Recall that $C_0^2 = - e, f^2 = 0, C_0 \cdot f = 1$ define the intersection theory on $X$. (For $e>0$ the rational curve $C_0$ is the only curve of negative self-intersection on $X$; in particular $h^0(X,C_0) = 1$, then; and $|C_0|$ is the second ruling of $X$ in the case $e=0$.)

Viewed on $X$ our quadrigonal curve $C$ is a smooth member of the linear series $|4C_0 + xf|$ for some integer $x$ if the ruling $|f|$ of $X$ cuts out on $C$ our chosen pencil $g_4^1$; note that $x \ge 4e$ (\cite{H}, V, 2.18), and that $x \ge 4$ for $e=0$ since $C$ is 4-gonal. The canonical series $|K_X|$ of $X$ is given by $|K_X| = |-2C_0 - (2+e)f|$ (\cite{H}, V, 2.11), and by adjunction we have $K_C \sim (K_X + C)|_C = (2C_0 + (x-e-2)f)|_C$ and so $2g-2 =$ deg$(K_C) =(2C_0 + (x-e-2)f) \cdot C = 2(x-4e) + 4(x-e-2) = 6x - 12e - 8$, i.e. $g = 3x - 6e - 3$. Moreover, $K_C \sim (2C_0 + (x-e-2)f)|_C = 2(C_0 + ef)|_C + (x - 3e - 2)f|_C$ with $|f|_C| = g_4^1$ and $x - 3e - 2 \ge 0$ since $x \ge 4e \ge 3e+2$ for $e \ge 2$, $x \ge 4 \ge 2 = 3e+2$ for $e=0$, and for $e=1$ we have $x \ge 4e = 4$, and $x < 3e+2 = 5$ would imply that $x=4$ and, then, $g = 3x - 6e - 3 = 3$, a contradiction. Furthermore, $h^0(K_X) = 0 = h^1(K_X)$ (\cite{H}, V, 2.5), and it follows that every canonical divisor of $C$ is cut out on $C$ by the adjoint system $|K_X + C|$ of $C$ on $X$, i.e. we have $(|K_X + C|)|_C = |(K_X + C)|_C| = |K_C| = |2 \Gamma + (x - 3e - 2)g_4^1|$ with the complete linear series $\Gamma := |(C_0 + ef)|_C|$ of degree $g-1-2(x-3e-2) = x < g$ on $C$.

Recall that the complete linear series $g_d^r$ on $C$ we started with is cut out on $C$ by $|C_0 + nf|$, with $n = \frac{r-1+e}{2} \ge e$. Let $h := n - e \ge 0$; then $g_d^r = |\Gamma  + (n-e)g_4^1| = |\Gamma + hg_4^1|$. We claim that $|C_0 + (n-i)f| = |\Gamma + (h-i)f|$ cuts out on $C$ the complete linear series $|g_d^r - ig_4^1| = |\Gamma + (h-i)g_4^1|$, for $i = 0, 1, ... , h+1$. In fact, let $i' > 0$ be the smallest of these integers $i$ which violates this claim. Since dim$|C_0 + (n-i)f| = 2(n-i) - e + 1 = r - 2i$ for $i = 0, 1, ... , h$ (\cite{H}, V, 2.19) we have dim$|g_d^r - (i' - 1)g_4^1| = r - 2(i' - 1)$  but dim$|g_d^r - i'g_4^1| \ge r - 2i' + 1$ whence dim$|(g_d^r - i'g_4^1) + g_4^1| =$ dim$|g_d^r - (i' - 1)|g_4^1| = r - 2i' + 2 \le $ dim$|g_d^r - i'g_4^1| + 1$. This implies (\cite{CM1}, 1.8) that the base point free part of $|g_d^r - (i' - 1)g_4^1|$ is a multiple of $g_4^1$ contradicting the fact that it contains the linear series cut out on $C$ by $|C_0 + (n - i' + 1)f|$. Hence we see that $|g_d^r - ig_4^1|$ is cut out on $C$ by $|C_0 + (n-i)f|$, for $i = 0, 1. ... , h+1$, and the latter series on $X$ is very ample for $i = 0, ... , h-1$ (note that $n-i > e$, then) thus inducing on $C$ complete and very ample linear series $g_{d-4i}^{r-2i} = |g_d^r - ig_4^1|$. For $i=h$ we obtain $|C_0 + ef|$ on $X$ which induces $\Gamma$ on $C$, with dim$(\Gamma) = r - 2h = (2n - e + 1) - 2(n-e) = e+1$ and deg$(\Gamma) = (C_0 + ef) \cdot C = x$. Since $|C_0 + ef|$ is base point free but not very ample we see that $\Gamma = g_x^{e+1}$ is non-trivial and simple for $e>0$ (moving $C$ into a cone in ${\bP}^{e+1}$ over a rational normal curve in ${\bP}^e$ if $e \ge 2$), and for $e=0$ it is a non-trivial pencil which is not compounded of the same involution as our $g_4^1$ and is different from the $g_4^1$ for $x=4$ since it is induced then by the second ruling $|C_0|$ of $X$. The series $|\Gamma - g_4^1| = |g_d^r - (h+1)g_4^1|$ is induced by $|C_0 + (e-1)f|$ which latter series is empty for $e=0$ and has for $e>0$ the fixed curve $C_0$ as a base curve since $C_0 \cdot (C_0 + (e-1)f) = - 1$. Hence the series $|\Gamma - g_4^1|$ is empty for $e=0$, and for $e>0$ (observing that dim$|\Gamma - g_4^1| \le $ dim$(\Gamma) - 2 = e-1$) it equals $(e-1)g_4^1 + D_0$ with the fixed divisor $D_0 := C_0|_C$ of degree $x - 4e \ge 0$. Thus $\Gamma = |eg_4^1 + D_0| = g_x^{e+1}$ is a linear series of type 3 which is very ample if and only if deg$(D_0) = C_0 \cdot C \le 1$. (Note that $e=0$ implies that $\Gamma = |D_0|$ is a pencil; so deg$(D_0) \le 1$ is impossible, then. For $e=1$ we have $\Gamma = g_x^2 = |g_4^1 + D_0|$ which is only very ample for deg$(D_0) = 1$, i.e. $x = 5$; then $V = \lbrace g_5^2 \rbrace$ which we have excluded here.)

Let $r_0 := e+1 =$ dim$(\Gamma)$ and $d_0 := x =$ deg$(\Gamma)$. Then the series $|\Gamma + (x-3e-2)g_4^1| = |K_C - \Gamma|$ has dimension $g-1-x+(e+1) = r_0 + 2(x-3e-2)$ since $g = 3x - 6e - 3$, and this implies that dim$|\Gamma + ig_4^1| = r_0 + 2i$ for $i = 1, ... , x-3e-2$. Assume that one of these series is of type 2, i.e. $|\Gamma + ig_4^1| = |K_C - (g-1-(d_0+4i) + (r_0+2i))g_4^1 - F|$ for some $i$, $0 \le i \le x-3e-2$, and for some effective divisor $F$ of $C$. Then $|\Gamma + ig_4^1| = |2\Gamma - (x-2e-1-2i)g_4^1 - F|$ (recall that $|K_C| = |2\Gamma + (x-3e-2)g_4^1|$). Hence we obtain $\Gamma = |(x-2e-1-i)g_4^1 + F|$ with $x-2e-1-i \ge x-2e-1 - (x-3e-2) = e+1$, a contradiction. So all series $|\Gamma + ig_4^1|$ for $i = 1, ... , x-3e-2$ are in $V$, and $\Gamma$ is not of type 2. And since we observed already that $|\Gamma - g_4^1|$ is - if non-empty - of type 1 away from its base locus $D_0$ we see that $|\Gamma + (x-3e-1)g_4^1| = |K_C - (\Gamma - g_4^1)|$ is non-special or of type 2. In particular, our series $g_d^r$ we started with is contained in the set $\lbrace |\Gamma + ig_4^1| : i = 0, ... , x-3e-2 \rbrace$. 

\vspace{1ex}

According to \cite{CC}, Theorem 10.2 we have the relation $e = \frac{g}{3}-1 - e_1$ (thus interpreting $e$ as the deviation of $e_1$ from its maximum possible value $\frac{g}{3}-1$); in particular, all Hirzebruch surfaces containing $C$ have the same invariant $e$. Consequently, with $g_{d_0}^{r_0} := \Gamma$ we just have obtained the very ample series listed in the Theorem, since $r_0 = e+1 = \frac{g}{3}-e_1$ and $d_0 = x = \frac{g+6e+3}{3} = \frac{g}{3} +2e + 1 = g-1 - 2(\frac{g}{3} - e - 1) = g-1 - 2e_1$, as wanted. And we notice that the expression $x-3e-2$ we used throughout is nothing but the first scrollar invariant $e_1$ of $C$.\\

We still have to prove that $V$ is already exhausted by the series already obtained. To do this it remains to show that starting with another series $g_{d'}^{r'}$ in $V$ (instead of the $g_d^r$ we started with before) does not provide new series in $V$. The $g_{d'}^{r'}$ gives rise to a Hirzebruch surface containing $C$ which again has the invariant $e = \frac{g}{3}-1-e_1$, and then to series $|\Gamma' + ig_4^1|$ in $V$, for a "generating" series $\Gamma'$ of the same dimension $e+1$ and, therefore, the same degree $x$ as $\Gamma$. Hence it is enough to show that any two non-trivial $g_{d_0}^{r_0} = g_x^{e+1}$ on $C$ coincide if they share the geometric properties stated in the Theorem; in particular, $\Gamma$ is then the only generator of the whole set $V$.

As a first case we assume that $e>0$ and that $\Gamma$, $\Gamma'$ are two different non-trivial $g_x^{e+1}$ on $C$, both being simple. We use the general

\vspace{1ex}

{\it Claim}: {\it Let $L_1$, $L_2$ be complete, base point free and simple linear series on $C$. Then $|L_1 + L_2|$ is of type 2 if it is still non-trivial.}

\vspace{1ex}

To prove the Claim, let $E = P_1 + P_2 + P_3 + P_4$ be a general divisor in $g_4^1$ (in particular, these points $P_i$ of $C$ are pairwise different). Since $L_1$ and $L_2$ are base point free and simple the series $|L_1 - P_1|$ and $|L_2 - P_2|$ are base point free. Hence we can take $F_1 \in |L_1 - P_1|$, $F_2 \in |L_2 - P_2|$ such that $P_i \notin F_j$ for $i = 1, ... , 4$, $j = 1, 2$, and $F_1 \cap F_2 = \emptyset$. Let $d_j :=$ deg$(L_j)$. The complete linear series $|F_1+F_2+E| = |L_1+L_2+P_3+P_4|$ of degree $d_1+d_2+2$ can only have $P_3$ or $P_4$ as a base point; but, by our choice of $F_1$ and $F_2$, the pencil $F_1+F_2+g_4^1$ in $|F_1+F_2+E|$ has none of these two points as a base point. Hence $|L_1+L_2+P_3+P_4|$ is base point free which implies that dim$|L_1+L_2+P_3+P_4| \ge $ dim$|L_1+L_2| + 1$. The linear series $|L_1+L_2+P_2+P_3+P_4| = |L_1+L_2+(g_4^1 - P_1)|$ contains the subpencil $F_1+L_2+g_4^1$  which cannot have $P_2$ as a base point (since $P_2 \notin F_1$); hence $P_2$ is not a base point of $|L_1+L_2+P_2+P_3+P_4|$ which shows that dim$|L_1+L_2+P_2+P_3+P_4| \ge $ dim$|L_1+L_2+P_3+P_4| + 1 \ge $ dim$|L_1+L_2| + 2$. Since $|L_1+L_2+E| = |L_1+L_2+g_4^1|$ is base point free we see that dim$|L_1+L_2+g_4^1| = $ dim$|L_1+L_2+E| \ge $ dim$|L_1+L_2+P_2+P_3+P_4| + 1 \ge $ dim$|L_1+L_2| + 3$. By the Riemann-Roch theorem this means that dim$|K_C - (L_1+L_2+g_4^1)| \ge $ dim$|K_C - (L_1+L_2)| - 1$. If dim$|K_C - (L_1+L_2+g_4^1)| = $ dim$|K_C - (L_1+L_2)|$ then $|L_1 + L_2|$ is non-special. If dim$|K_C - (L_1+L_2+g_4^1)| = $ dim$|K_C - (L_1+L_2)| - 1$ then the base point free part of $|K_C - (L_1+L_2)|$ is $mg_4^1$ with $m := $ dim$|K_C - (L_1+L_2)| \ge 0$ (\cite{CM1}, 1.8), hence is zero resp. of type 1 according to $m=0$ resp. $m>0$. This proves the Claim.

\vspace{1ex}

We apply the Claim to $L_1 := \Gamma$ and $L_2 := \Gamma'$. Assume that $|\Gamma + \Gamma'|$ is special. Then, by the Claim, $|K_C - (\Gamma + \Gamma')| = mg_4^1 + F$ with $m := $ dim$|K_C - (\Gamma + \Gamma')| \ge 0$ and $F$ an effective divisor of $C$. In particular, $2g-2 - 2x = $ deg$(K_C - (\Gamma + \Gamma')) \ge 4m \ge 4(g-1-2x$ + dim$|\Gamma + \Gamma'|)$, and so 4dim$|\Gamma + \Gamma'| \le 6x-2g+2 = 6x - 2(3x-6e-3) + 2 = 12e+8$, i.e. dim$|\Gamma + \Gamma'| \le 3e+2$. But since $\Gamma \ne \Gamma'$ we have dim$|\Gamma + \Gamma'| \ge 3(e+1)$, according to \cite{A}, 5.1. This is a contradiction. Hence $|\Gamma + \Gamma'|$ is non-special and so dim$|\Gamma + \Gamma'| = 2x-g = 2x-(3x-6e-3) = 6e+3-x \le 2e+3$ since $x \ge 4e$. Since we assume $e>0$ this again contradicts \cite{A}, 5.1.

\vspace{1ex}

So we are left with the case $e=0$, as the last step. So let $g = 3(x-1)$, and assume that $g_x^1$ and $h_x^1$ are two different complete and base point free pencils of degree $x$ on $C$ which are different from our chosen $g_4^1$ and are both not compounded of the same involution as the $g_4^1$.

Using the very ample web $g_{x+4}^3 = |g_x^1 + g_4^1|$ we see that $C$ is a curve of type (or bi-degree) $(x,4)$ on a smooth quadric $Q \cong X_0$ in ${\bP}^3$, and the two rulings $|C_0|$ resp. $|f|$, say, of $Q$ cut out on $C$ the pencils $g_x^1$ resp. $g_4^1$. The canonically adjoint series $|K_Q + C| = |2C_0 + (x-2)f|$ of $C$ on $Q$ contains the curves made up by two lines in $|C_0|$ plus $x-2 \ge 2$ lines in $|f|$. Concerning the pencil $h_x^1$ we denote by $m>1$ the degree of the covering $\varphi$ if $g_x^1$, $h_x^1$ are compounded of the same involution $\varphi : C \rightarrow C'$, and we let $m=1$ if $g_x^1$, $h_x^1$ are not compounded of the same involution.

Take a general element $E = P_1 + ... + P_x$ in $h_x^1$ (in particular, these points $P_i$ of $C$ are pairwise different). Since $h_x^1$ and $g_4^1$ are not compounded of the same involution each point $P_i$ is contained in exactly one line $L_i \in |f|$ ($i = 1, ... , x$) whereas there are $\frac{x}{m}$ lines $R_m$, $R_{2m}$, $R_{3m}$, ... , $R_{x-m}$, $R_x$ in $|C_0|$ such that $R_{\lambda m}$ contains the $m$ points $P_{(\lambda -1)m+1}, ... ,P_{\lambda m}$ ($\lambda = 1 , ... , \frac{x}{m}$). Furthermore, let $L_0$ resp. $R_0$ denote a line in $|f|$ resp. in $|C_0|$ containing none of the points $P_1, ... , P_x$.

Then the curve $(2m+i-3)L_0 + L_i + ... + L_{x-2m} + R_{x-m} + R_x \in |(x-2)f + 2C_0|$ obviously lies in $|(K_Q + C)(- P_i - ... - P_x)|$ (the linear subseries of $|K_Q + C|$ made up by those adjoint curves of $C$ which pass through the points $P_i, ... , P_x$) but not in $|(K_Q + C)(- P_{i-1} - P_i - ... - P_x)|$, for $i = 2, ... , x-2m+1$. Hence we have $|(K_Q + C)(-E)| \subsetneqq |(K_Q + C)(-P_2 - ... - P_x)| \subsetneqq ... \subsetneqq |(K_Q + C)(- P_{x-2m+1} - ... - P_x)|$.

The adjoint curve $(x-m+j-3)L_0 + L_{x-2m+j} + ... + L_{x-m} + R_0 + R_x$ lies in $|(K_Q + C)(- P_{x-2m+j} - ... - P_x)|$ but not in $|(K_Q + C)(- P_{x-2m+j-1} - ... - P_x)|$ for $j = 2, ... , m+1$. Last, the adjoint curve $(x-m+j-3)L_0 + L_{x-m+j} + ... + L_x + 2R_0$ lies in $|(K_Q + C)(- P_{x-m+j} - ... - P_x)|$ but not in $|(K_Q + C)(- P_{x-m+j-1} - ... - P_x)|$, again for $j = 2, ... , m+1$. (For $j = m+1$ we end up with the adjoint curve $(x-2)L_0 + 2R_0$.)

Consequently, the $x$ points $P_1, ... , P_x $ supporting $E \in h_x^1$ impose independent linear conditions on the canonical adjoint curves of $C$, and since $(|K_Q + C|)|_C = |(K_Q + C)|_C| = |K_C|$ it follows that dim$|K_C - E| = g-1-x$. But by the Riemann-Roch theorem we have dim$|K_C - E| = g-1-x$ + dim$|E| = g-x$. This contradiction finally proves the Theorem. 
     
\end{proof}

{\it Remark}: The Theorem can partly be generalized to $k$-gonal curves $C$ ($k>4$) if we restrict thereby to those non-trivial and very ample linear series $g_d^r$ which satisfy dim$|g_d^r - g_k^1| = r-2$. But this provides no longer a complete decription of the complete and very ample series on $C$, of course.

\section{Consequences}

The Theorem shows that $V$ is (for fixed $g$) completely determined by the scrollar invariant $e_1$ of $C$, and that $V$ is "as small as possible".                                              
In particular:

\begin{corollary} \label{Cor1}
If a quadrigonal curve $C$ of genus $g$ has a complete, special and very ample $g_d^r$ such that $3d < 2g-2+4r$ then this series is unique.
\end{corollary}

\begin{proof}
For $r = d-g+1$ the inequality means $d > 2g-2$ which contradicts the Riemann-Roch theorem. Hence the $g_d^r$ is non-trivial but not of type 2 whence it lies in $V$, and $V$ contains for a fixed degree at most one linear series.
\end{proof}

\begin{corollary} \label{Cor2}
For a quadrigonal curve $C$ of genus $g$ with $V \ne \emptyset$ its scrollar invariant satisfies $e_1 \ge \frac{g-9}{6}$, and we have $e_1 = \ulcorner \frac{g-9}{6} \urcorner$ if and only if the generator of $V$ lies in $V$. 
\end{corollary}

\begin{proof}
Let $g_{d_0}^{r_0} =|(r_0 - 1)g_4^1 + D_0|$ be the generator of $V$. Then deg$(D_0) = d_0 - 4(r_0 - 1) = (g-1-2e_1) - 4(\frac{g}{3} - e_1 - 1) = 3 - \frac{g}{3} + 2e_1$. Since deg$(D_0) \ge 0$ we obtain $e_1 \ge \frac{g-9}{6}$, and then deg$(D_0) \le 1$ iff $e_1 = \ulcorner \frac{g-9}{6} \urcorner$.
\end{proof}

\begin{corollary} \label{Cor3}
Let $C$ be quadrigonal with $V \ne \emptyset$. Then a complete and very ample linear series of minimal degree on $C$ lies in $V$; in particular, it is unique.
\end{corollary}

\begin{proof}
By Example 1 we may assume that $g>9$. Then $e_1 \ge 1$ (by Corollary \ref{Cor2}), and $V$ contains a series of degree $d_0 + 4 = g + 3 - 2e_1 \le g+1$. Hence a complete and very ample series $g_d^r$ of minimal degree has dimension $r \ge 3$ and degree $d \le d_0 + 4 \le g+1$ and must therefore be non-trivial. Assume that this series is not in $V$. Then it is of type 2 whence $3d \ge 2g-2+4r$.

If the generator of $V$ is in $V$ then $d \le d_0 = g-1-2e_1 \le g-1 - 2\frac{g-9}{6} =\frac{2g}{3} + 2$, i.e. $3d \le 2g+6 < 2g-2+4r$, a contradiction.

So we may assume that the generator of $V$ is not in $V$. Then $e_1 > \ulcorner \frac{g-9}{6} \urcorner$, i.e. $e_1 \ge \frac{g-3}{6}$, and $V$ contains a series of degree $d_0 + 4 = g+3-2e_1$. For $e_1 \ge \frac{g}{6}$ we obtain $3d \le 3(d_0 + 4) \le 2g+9 < 2g-2+4r$, a contradiction again. For $e_1 = \frac{g-3}{6}$ (then $g$ is odd) we obtain $3d \le 3(d_0 + 4)  = 2g+12 < 2g-2+4r$ unless $r=3$. So we are done for $(r, e_1) \ne (3, \frac{g-3}{6})$.

For the remaining case $r=3$, $e_1 = \frac{g-3}{6}$ we have to go into the details. In that case we have $2g+12 = 3(d_0 + 4) \ge 3d \ge 2g-2+4r = 2g+10$ which implies that $d = d_0 + 4 = \frac{2}{3}g + 4$. Since $g_d^r = g_d^3$ is of type 2 it follows that $g_d^3 = |K_C - (\frac{g}{3} - 2)g_4^1 - F|$ for an effective divisor $F$ of $C$ of degree deg$(F) = 2$ whereas $g_{d_0 + 4}^{r_0 + 2} = |g_{d_0}^{r_0} + g_4^1| = |\frac{g+3}{6}g_4^1 + D_0|$ with $r_0 = \frac{g}{3} - e_1 = \frac{g+3}{6}$ and deg$(D_0) = 2$. By the Theorem, $|K_C| =|2g_{d_0}^{r_0} + e_1 g_4^1| = |g_{d_0 + 4}^{r_0 + 2} + g_{d_0}^{r_0} + (e_1 - 1)g_4^1| = 
|g_{d_0 + 4}^{r_0 + 2} + ((r_0 - 1)g_4^1 + D_0) + (e_1 - 1)g_4^1|$ whence $g_d^3 = |g_{d_0 + 4}^{r_0 + 2} + (r_0 + e_1 - 2 - (\frac{g}{3} - 2))g_4^1 + D_0 - F| = |g_{d_0 + 4}^{r_0 + 2} + D_0 - F|$ since $r_0 + e_1 - \frac{g}{3} = 0$. But since $g_{d_0 + 4}^{r_0 + 2}$ and $g_d^3$ are very ample we obtain $r_0 =$ dim$|g_{d_0 + 4}^{r_0 + 2} - F| =$ dim$|g_d^3 - D_0| = 1$, i.e. $\frac{g+3}{6} = r_0 = 1$, a contradiction. 
\end{proof}

The following result extends Example 1.

\begin{corollary} \label{Cor4}
Let $C$ be a curve of genus $g > 9$. If $C$ is quadrigonal with $V \ne \emptyset$ then $C$ has a complete and very ample $g_{g-3}^{\frac{g}{3}-1}$ resp. $g_{g-1}^{\frac{g}{3}}$ according to the scrollar invariant $e_1$ of $C$ is odd resp. even, and this series embeds $C$ as an extremal curve; in particular, for even $e_1$ this series is half-canonical. 

Conversely, let $3 | g$, and if $g=15$ assume that $C$ is not isomorphic to a smooth plane septic. Then the existence of a simple $g_{g-3}^{\frac{g}{3}-1}$ resp. $g_{g-1}^{\frac{g}{3}}$ on $C$ implies that $C$ is quadrigonal, and this series is in $V$.
\end{corollary}

\begin{proof}
As for the first claim, assume that $e_1 \ge 2$. Then we can choose $i = \frac{e_1 - 1}{2} \ge 1$ (if $e_1$ is odd) resp. $i = \frac{e_1}{2}$ (if $e_1$ is even) in the Theorem to obtain a $g_{g-3}^{\frac{g}{3}-1}$ resp. $g_{g-1}^{\frac{g}{3}}$ in $V$, and by Castelnuovo's genus bound (e.g. \cite{A}, 3.3) this series embeds $C$ as an extremal curve since $g \ge 12$. By the Theorem (or by \cite{ACGH}, III, 2.6) the $g_{g-1}^{\frac{g}{3}}$ is half-canonical. Assume that $e_1 = 0$. Then, by Corollary \ref{Cor2}, $g \le 9$, a contradiction. Assume that $e_1 = 1$. Then $\char35 V \le e_1 + 1 = 2$. If $\char35 V = 1$ we know from Example 1 that $g \le 9$, a contradiction. If $\char35 V = 2$ then the generator $g_{g-1-2e_1}^{\frac{g}{3}-e_1} = g_{g-3}^{\frac{g}{3}-1}$ of $V$ lies in $V$, and we observed just before that this series embeds $C$ into ${\bP}^{\frac{g}{3}-1}$ as an extremal curve of degree $g-3$. (But this time the embedded curve lies on a cone of degree $\frac{g}{3}-2$, and Corollary \ref{Cor2} implies that $g=12$ or $g=15$.)

Concerning the converse, note that we already noticed that for a simple linear series $g_d^r$ with $(d,r) = (g-3, \frac{g}{3}-1)$ resp. $(d,r) = (g-1, \frac{g}{3})$ on $C$ the related Castelnuovo's genus bound equals $g$ whence this series actually embeds $C$ into ${\bP}^r$ as an extremal curve of degree $d$. Then $C$ lies on a surface scroll of degree $r-1$ in ${\bP}^r$ or, for $r=5$, possibly on a Veronese surface $V_4$ in ${\bP}^5$. For $r=5$ we have $g=18$ and $d=15$ resp. $g=15$ and $d=14$, and $C$ can only lie on $V_4$ for even $d$ whence $g_d^r = g_{14}^5 = |2g_7^2|$, and $C$ is a smooth plane septic (and therefore 6-gonal), then. Since we excluded this case we see that $C$ lies on a scroll $S$ of degree $r-1$ in ${\bP}^r$. Let $m := [\frac{d-1}{r-1}]$; then $d-1 = m(r-1) + \epsilon$ for some integer $\epsilon$ with $0 \le \epsilon \le r-2$. It is well known that the extremal curve $C \subset S$ has gonality gon$(C) = m+1$ if $\epsilon \ne 0$, and that $m \le$ gon$(C) \le m+1$ for $\epsilon = 0$. In our situation (i.e. for our values of $d$ and $r$) one computes that $m=3$ and $\epsilon \ne 0$ unless $g=12$ and $g_d^r = g_9^3$. Hence if $g_d^r \ne g_9^3$ it follows that gon$(C) = 4$. If $g_d^r = g_9^3$ then $C$ is via this series a smooth curve of degree $d=9$ on a quadric $S$ in ${\bP}^3$; if $S$ is smooth $C$ is a curve of type $(4,5)$, and if $S$ is a cone the projection $C \rightarrow {\bP}^1$ from the vertex of $S$ has degree $\frac{d-1}{deg(S)} = 4$. Thus $C$ is quadrigonal also in the case $(d,r) = (9,3)$. Finally, we observe that $3d < 2g-2+4r$ whence the $g_d^r$ is not of type 2.
\end{proof}

A similar (even simpler, though less precise) description of the quadrigonal curves with $V \ne \emptyset$ is the

\begin{corollary} \label{Cor7}
Let $C$ be a quadrigonal curve. Then $V \ne \emptyset$ if and only if $C$ is an extremal curve of degree $d > 2r$ in ${\bP}^r$, for some integer $r \ge 2$.
\end{corollary}

\begin{proof}
If $V \ne \emptyset$ then $3|g$, and our assertion follows from Corollary \ref{Cor4} resp. from Example 1. Conversely, we prove the

\vspace{1ex} 

{\it Claim}: {\it Let $C$ be a quadrigonal curve which is extremal of degree $d > 2r$ in ${\bP}^r$ (for some $r \ge 2$). Then the corresponding $g_d^r$ is in $V$.}

\vspace{1ex}

In fact, $C$ is for $r=2$ a smooth plane quintic since we assume gon$(C) = 4$. If $C$ is a curve of degree $d$ on a Veronese surface in ${\bP}^5$ then $C$ is a smooth plane curve of degree $\frac{d}{2} > r = 5$ whence gon$(C) = \frac{d}{2} - 1 > 4$, a contradiction. So $C$ lies on a scroll of degree $r-1$ in ${\bP}^r$ ($r \ge 3$). Let $m := [\frac{d-1}{r-1}]$, i.e. $d-1 = m(r-1) + \epsilon$ for an integer $\epsilon$ with $0 \le \epsilon \le r-2$. Then we have $4 =$ gon$(C) = m+1$ if $\epsilon \ne 0$, and $m \le$ gon$(C) = 4 \le m+1$ if $\epsilon = 0$. So $m \le 4$, and $m=3$ for $\epsilon \ne 0$.

Let $m=4$. Then $\epsilon = 0$, $d-1 = 4(r-1)$, and $g = m(d-1 - \frac{1}{2}(m+1)(r-1)) = 6(r-1)$, cliff$(g_d^r) = d-2r = 2r-3 = \frac{g}{3} - 1$.

Let $m=3$. Then $d-1 = 3(r-1) + \epsilon$, with $0 \le \epsilon \le r-2$, and we obtain $g = 3(d-1 - 2(r-1)) = 3(d-2r+1)$, cliff$(g_d^r) = d-2r = \frac{g}{3} - 1$.

Assume that the embedding series $g_d^r$ is of type 2. Then $g-3 = 3$cliff$(g_d^r) = 3d-6r \ge (2g-2+4r) - 6r = 2g-2-2r$, i.e. $2r > g$. For $m=4$ it follows that $2r > 6r-6$, i.e. $4r < 6$, a contradiction. For $m=3$ we see that $2r > 3(d-2r+1) = 3(r-1 + \epsilon)$ whence $3 \epsilon < 3-r \le 0$, a contradiction again. Hence $g_d^r \in V$.
\end{proof}

\begin{corollary} \label{Cor5}
Let $C$ be quadrigonal of genus $g>9$. If $V \ne \emptyset$ then $C$ lies in a locus of codimension $\frac{g}{3}$ in the moduli space $M_g(4)$ of quadrigonal curves of genus $g$.
\end{corollary}

\begin{proof}
Using the two very ample series in Corollary \ref{Cor4} the result follows from the last formula in \cite{A}, section 6 (being concerned with the number of moduli of extremal curves), and by taking into account that dim$(M_g(4)) = 2g+3$.
\end{proof}

Corollary \ref{Cor5} gives a precision of the main result 3.3 in \cite{M2}. By Example 1 it is also true for $g \le 9$ unless $C$ is a plane sextic with a single double point ($g=9$).

\vspace{1ex}

Finally, we exhibit a rather special class of quadrigonal curves for which $V$ is empty:

\begin{corollary} \label{Cor6}
Let $C$ be a double covering of genus $g \ge 4h$ over a hyperelliptic curve of genus $h>1$. Then $C$ is quadrigonal, and all nontrivial and very ample linear series on $C$ are of type 2.
\end{corollary}

\begin{proof}
Let $k$ denote the gonality of $C$. If $k=3$ then, by observing that $C$ admits two morphisms, of degree 3 resp.4, onto a rational curve, we obtain that $g \le 6$ contradicting $g \ge 4h \ge 8$. If $k=2$ then $C$ is a double covering of ${\bP}^1$ and also of a curve of genus $h \ne 0$ which implies (\cite{ACGH}, VIII, ex. C-1) that $g \le 1 + 2h$ which again contradicts $g \ge 4h$. Hence we have $k=4$.

Assume that $V \ne \emptyset$; then $3|g$. Since $e_1 \le \frac{g}{3} - 1$ we can write $e_1 = \frac{g}{3} - 1 - j$ with $0 \le j \in {\bZ}$. By the Theorem, the generator $g_{d_0}^{r_0}$ is complete and base point free of degree $d_0 = g-1-2e_1$ and dimension $r_0 = \frac{g}{3} - e_1 = j+1$. Let $j \ge 1$, i.e. $r_0 \ge 2$. Then $g_{d_0}^{r_0}$ is simple, and by subtracting $j-1$ general points of $C$ from it we obtain a simple net of degree $d_0 - (j-1) = g - 2e_1 - j = g - 2e_1 - (\frac{g}{3} -1 - e_1) = \frac{2g}{3} - e_1 + 1$. According to \cite{HKO}, Theorem C the smallest degree $s(C,2)$ of a simple net on $C$ satisfies $s(C,2) \ge g - 2h + 3$. Consequently, $g - 2h + 3 \le s(C,2) \le \frac{2g}{3} - e_1 + 1$, i.e. $g \le 6h-6-3e_1$, and so from Corollary \ref{Cor2} it follows that $g \le 4h-1$, a contradiction. Let $j=0$, i.e. $r_0 = 1$, $d_0 = \frac{g}{3} + 1$. Then $|g_{d_0}^{r_0} + g_4^1|$ is a very ample $g_{\frac{g}{3} + 5}^3$ whence (by subtracting one point of $C$ from it) we see that $g - 2h + 3 \le s(C,2) \le \frac{g}{3} + 4$. But then $8h \le 2g \le 6h+3$, i.e. $h < 2$, a contradiction again.  
\end{proof}

Observe that there are smooth plane quintics (so $V \ne \emptyset$) which doubly cover a curve of genus 2. An example is the Fermat quintic $x^5 + y^5 = 1$ which obviously admits an automorphism of order 2 by interchanging $x$ and $y$.\\

\end{document}